\documentclass[12pt]{amsart}
\usepackage{amsmath,amssymb,amsthm,mathtools}
\usepackage{mathrsfs,bm}
\usepackage{thmtools,thm-restate}
\usepackage{tikz-cd}
\usepackage{tikz}
\usepackage{enumitem}
\usepackage{graphicx}
\usetikzlibrary{matrix,arrows.meta,bending}
\usepackage{color}
\usepackage[numbers]{natbib}
\usepackage{booktabs}
\usepackage{longtable}
\usepackage{pgfplots}
\pgfplotsset{compat=1.14} 
\usetikzlibrary{arrows.meta,decorations.pathreplacing,positioning,calc} 

\usepackage[a4paper, left=3cm, right=3cm, top=3cm, bottom=3cm, footskip=15mm]{geometry}

\newtheorem{theorem}{Theorem}[section]

\newtheorem{corollary}[theorem]{Corollary}

\theoremstyle{definition}

\theoremstyle{remark}
\newtheorem{remark}[theorem]{Remark}

\numberwithin{equation}{section}

\makeindex

\usepackage{fancyhdr}
\usepackage{calc} 

\AtBeginDocument{%
  \setlength{\textwidth}{\dimexpr\paperwidth - 6cm\relax}%
  \setlength{\oddsidemargin}{\dimexpr 3cm - 1in\relax}%
  \setlength{\evensidemargin}{\dimexpr 3cm - 1in\relax}%
  \setlength{\marginparwidth}{2.5cm}%
}

\begin{document}

\title{On the pointwise periodicity of multiplicative and additive functions}

\author{T. Agama}
\address{Department of Mathematics, African Institute for Mathematical science, Ghana
}
\email{theophilus@aims.edu.gh/emperordagama@yahoo.com}


\date{\today}



\begin{abstract}
We study the problem of estimating the number of points of coincidences of an idealized gap on the set of integers under a given multiplicative function $g:\mathbb{N}\longrightarrow \mathbb{C}$, respectively, an additive function $f:\mathbb{N}\longrightarrow \mathbb{C}$. We obtain various lower bounds depending on the length of the period, by varying the worst growth rates of the ratios of their consecutive values.
\end{abstract}

\maketitle

\begingroup
  \setlength{\parskip}{6pt} 
  \tableofcontents
\endgroup

\section{Introduction}

Additive and multiplicative arithmetic functions occupy a central place in modern analytic and probabilistic number theory. They arise naturally from the prime factorization of integers and encode subtle information about the arithmetic structure of positive integers through their values on prime powers and coprime products. Classical references such as \cite{apostol1976,elliott1985,galambos1970,tenenbaum2015introduction,May} show that these classes of functions are not only fundamental in their own right, but also serve as a bridge between average-order phenomena, extremal-order estimates, and probabilistic models for arithmetic data (positive integers).\\

The present paper deals with a shifted coincidence problem for such functions. Given a fixed shift $l>0$, we study the size of the set of integers $n\leq x$ for which a prescribed arithmetic function repeats its value at the translated points $n-l$ and $n+l$. In the notation of the paper, these sets are measured by the quantities
$$
\mathcal{G}(x,l)_f,\quad \mathcal{G}(x,l)^{+}_f,\quad \mathcal{G}(x,l)^{-}_f,
$$
and their multiplicative analogs. Conceptually, these counting functions measure a weak form of periodic behavior: instead of asking whether the function is globally periodic, we ask how often a fixed translation creates a coincidence of values. This is a local and pointwise notion of recurrence, and it is particularly natural for arithmetic functions whose values are highly structured but not periodic in any classical sense.\\

A key theme of the paper is that the size of these coincidence sets is governed by the worst-case growth of consecutive values. Roughly speaking, if the ratio between neighboring values is small, then coincidences become more frequent; if that ratio grows faster, the lower bounds weaken accordingly. The paper therefore organizes its main results according to different growth regimes for
$$
\frac{g(n+1)}{g(n)} \quad\text{and}\quad \frac{f(n+1)}{f(n)},
$$
where $g$ is multiplicative and $f$ is additive by convention. This perspective leads to a uniform strategy: a shifted coincidence $f(n)=f(n+l)$ or $g(n)=g(n+l)$ is first restricted to the case $l\mid n$, and then rewritten in terms of consecutive values after substitution $n=l m$. The arithmetic information is thus transferred from a fixed shift to a comparison of neighboring values, where the assumed growth control can be directly exploited. This reduction is the mechanism behind all the lower bounds, as proved later in the paper.\\

The first family of results treats multiplicative functions with ratios of consecutive values bounded above by a logarithmic factor. In that regime, the coincidence count is shown to be at least on the scale of
$$
\frac{x/l}{\log\log(x/l)}.
$$
The paper then specializes this estimate to classical functions such as the Euler totient function $\varphi$ and the sum-of-divisors function $\sigma$, using their standard extremal-order bounds. A second multiplicative result shows that when the consecutive ratio is bounded above and below by absolute constants, the set of coincidences is essentially of full order $x/l$. This is the strongest regime considered in the paper, and it reflects the fact that near-constant multiplicative growth leaves little room for separation between neighboring values.\\

The additive side of the argument is analogous, but the bounds change according to the larger typical size of additive arithmetic functions. When the ratio of consecutive values grows no faster than a logarithm divided by a power of $\log\log n$, the paper proves a lower bound of order
$$
\frac{x}{l}\frac{(\log\log(x/l))^{c}}{\log(x/l)}.
$$
In the completely additive setting, a pure logarithmic ratio bound yields the more familiar scale
$$
\frac{x/l}{\log(x/l)}.
$$
These estimates are then illustrated by applying the general theorem to $\omega(n)$ and $\Omega(n)$, whose standard extremal behavior is well documented in the literature. Finally, the paper notes that if one allows for a linear growth rate for the ratio of consecutive values, the count of coincidence still cannot be too small; in that case, a logarithmic lower bound survives.\\

The preliminary section collects important facts about the extremal order of the arithmetic functions needed later in the paper, including standard estimates for $\varphi$, $\tau$, $\Omega$, $\sigma$, and $\omega$. These are not merely auxiliary tools: they provide concrete examples showing how the abstract ratio conditions translate into familiar arithmetic functions. The main section (lower-bound section) then develops the general coincidence estimates and derives the corollaries for specific functions. In this way, the paper moves from a broad structural principle--growth control implies coincidence abundance--to concrete arithmetic consequences for some of the most classical functions in number theory.

\subsection{Organization of the paper} The organization of the paper is as follows. Section~2 states the coincidence problem in a precise way and fixes the notation used throughout. Section~3 reviews the standard results of the extremal-order for the arithmetic functions that will later serve as examples. Section~4 contains the main lower-bound theorems, together with their corollaries and remarks.

\section{The problem statement}

Let $f:\mathbb{N}\longrightarrow \mathbb{C}$. We say that $f$ is periodic on the set $\{1,\ldots,x\}\subset \mathbb{N}$ with the period $l$ if $f(n-l)=f(n)=f(n+l)$ for all $n\in \{1,\ldots,x\}\subset \mathbb{N}$. We say that $f$ is pointwise left-periodic if for any $l>0$ there exists some $n\in \{1,\ldots, x\}\subset \mathbb{N}$ such that $f(n-l)=f(n)$. Similarly, we say that it is pointwise right-periodic if for any $l>0$ there exists some $n\in \{1,\ldots,x\}\subset \mathbb{N}$ such that $f(n)=f(n+l)$. We call $l:=l(n)>0$ the pointwise period. We say that it is fully-pointwise periodic with period $l:=l(n)>0$ if $f(n-l)=f(n)=f(n+1)$. In this paper, we study the problem of estimating the size of the quantity 
\begin{align}
\mathcal{G}(x,l)_f:=\# \{n\leq x:f(n-l)=f(n)=f(n+l),~\mathrm{for~fixed}~l>0\}.\nonumber
\end{align}
First, we obtain a general theorem for problems of this flavour and narrow it to specific examples by varying the arithmetic functions. We study the problem in the context of pointwise left-periodicity and pointwise right-periodicity. In particular, under a given multiplicative function $g$ or an additive function $f$, we study the size of the following quantities
\begin{align}
\mathcal{G}(x,l)^{+}_f:=\# \{n\leq x:f(n)=f(n+l),~\mathrm{for~fixed}~l>0\}\nonumber
\end{align}
and 
\begin{align}
\mathcal{G}(x,l)^{-}_f:=\# \{n\leq x:f(n-l)=f(n),~\mathrm{for~fixed}~l>0\},\nonumber
\end{align}
respectively, 
\begin{align}
\mathcal{G}(x,l)^{+}_g:=\# \{n\leq x:g(n)=g(n+l),~\mathrm{for~fixed}~l>0\}\nonumber
\end{align}
and 
\begin{align}
\mathcal{G}(x,l)^{-}_g:=\# \{n\leq x:g(n-l)=g(n),~\mathrm{for~fixed}~l>0\}.\nonumber
\end{align}
In particular, we obtain the following results 

\begin{theorem}
Let $g:\mathbb{N}\longrightarrow \mathbb{C}$ be a multiplicative function with $g(t)\neq 0$ for all $t\in \mathbb{N}$ such that \begin{align}
1\ll_g\frac{g(n+1)}{g(n)}\ll_g C\log \log n\nonumber
\end{align}
for an absolute constant $C>0$. We have 
\begin{align}
\mathcal{G}(x,l)^{+}_{g}\gg_g \frac{1}{C}\frac{\frac{x}{l}}{\log \log (\frac{x}{l})}.\nonumber
\end{align} 
\end{theorem}

\begin{theorem}
Let $f:\mathbb{N}\longrightarrow \mathbb{C}$ be a completely additive function with $f(n)\neq 0$ for all $n\geq 2$ such that \begin{align}
1\ll_f \frac{f(n+1)}{f(n)}\ll_f C\log n.\nonumber
\end{align}
We have the lower bound 
\begin{align}
\mathcal{G}(x,l)^{+}_{f}\gg_f \frac{1}{C}\frac{\frac{x}{l}}{\log (\frac{x}{l})}.\nonumber
\end{align}
\end{theorem}

\begin{theorem}
Let $g:\mathbb{N}\longrightarrow \mathbb{C}$ be a multiplicative function with $g(t)\neq 0$ for all $t\in \mathbb{N}$ such that \begin{align}
1\ll_g \frac{g(n+1)}{g(n)}\ll_g 1.\nonumber
\end{align}
We have
\begin{align}
\mathcal{G}(x,l)^{+}_{g}\gg_g \frac{x}{l}.\nonumber
\end{align}
\end{theorem}

\subsection{The notations}

Throughout this paper, we consider arithmetic functions $f,g:\mathbb{N}\longrightarrow \mathbb{C}$, where by convention $g$ denotes multiplicative (resp. completely multiplicative) functions and $f$ denotes additive (resp. completely additive) functions. We keep the standard notation: $f\ll h\Leftrightarrow |f(n)|\leq Ch(n)$ for all $n\geq n_0$ for some $C>0$, and similarly $f\gg h\Leftrightarrow |f(n)|\geq Kh(n)$ for all $n\geq n_0$ for some $n_0>0$ and some absolute constant $K>0$. In the case where the absolute constant is determined by some function $T$, we denote these respective inequalities by $f\ll_T h$ and $f\gg_T h$. We denote the quantity $\mathcal{G}(x,l)^{+}_g:=\# \{n\leq x:g(n)=g(n+l),~\mathrm{for~fixed}~l>0\}$ and similarly $\mathcal{G}(x,l)^{-}_{g}:=\# \{n\leq x:g(n-l)=g(n),~\mathrm{for~fixed}~l>0\}$.

\section{Preliminary results}

In this section, we review the theory of \emph{extremal orders} for arithmetic functions. We revisit, as is essential in our studies, the notion of the \emph{maximal} and the \emph{minimal} orders of various arithmetic functions. We then leverage these concepts in the sequel to establish particular examples of the main results of this paper.

\begin{theorem}\label{euler}
Let $\varphi(n):=\sum \limits_{\substack{m\leq n\\(m,n)=1}}1$ denote the Euler totient function. We have 
\begin{align}
\varphi(n)<n \nonumber 
\end{align}
and 
\begin{align}
\varphi(n)\gg e^{-\gamma}\frac{n}{\log \log n}\nonumber
\end{align}
where $\gamma$ is the Euler-Macheroni constant.
\end{theorem}

\begin{proof}
For a proof, see, e.g, \cite{tenenbaum2015introduction}.
\end{proof}

\begin{theorem}\label{divisor1}
Let $\tau(n):=\sum \limits_{d|n}1$ count the number of divisors of $n$. For $n\geq 1$, we have 
\begin{align}
\tau(n)\ll n^{\frac{\log 2}{\log\log n}}\nonumber
\end{align}
and $\tau(n)\geq 2$.
\end{theorem}

\begin{proof}
For a proof, see, e.g., \cite{tenenbaum2015introduction}.
\end{proof}

\begin{theorem}\label{multiplicity}
Let $\Omega(n)=\sum \limits_{p||n}1$ count the prime divisors of $n$ with multiplicity. For $n\geq 1$, we have $\Omega(n)\geq 1$ and 
\begin{align}
\Omega(n)\ll \frac{\log n}{\log 2}.\nonumber
\end{align}
\end{theorem}

\begin{proof}
For a proof, see, e.g., \cite{tenenbaum2015introduction}.
\end{proof}

\begin{theorem}\label{sumdivisor}
Let $\sigma(n):=\sum \limits_{d|n}d$ count the number of divisors of $n$. We have $\sigma(n)\geq n$ and 
\begin{align}
\sigma(n)\ll e^{\gamma}n\log\log n.\nonumber
\end{align}
\end{theorem}

\begin{proof}
For a proof, see, for example, the book of Tenenbaum \cite{tenenbaum2015introduction}.
\end{proof}

\begin{theorem}\label{divisor2}
Let $\omega(n):=\sum \limits_{p|n}1$ count the number of distinct prime divisors of $n$. We have the inequality 
\begin{align}
1\leq \omega(n)\ll \frac{\log n}{\log\log n}.\nonumber
\end{align}
\end{theorem}

\begin{proof}
For a proof, see, e.g., \cite{May}.
\end{proof}

\section{Lower bound}

In this section, we study the underlying problem in the context of functions $f:\mathbb{N}\longrightarrow \mathbb{C}$ with the property $f(mn)=f(n)+f(m)$ and those of the form $g:\mathbb{N}\longrightarrow \mathbb{C}$ with the property $g(mn)=g(m)g(n)$. We estimate from below the size of each of these sets under a given arithmetic function.

\begin{theorem}\label{Type1}
Let $g:\mathbb{N}\longrightarrow \mathbb{C}$ be a multiplicative function with $g(t)\neq 0$ for $t\in \mathbb{N}$ such that \begin{align}
1\ll_g\frac{g(n+1)}{g(n)}\ll_g C\log \log n\nonumber
\end{align}
for an absolute constant $C>0$. We have 
\begin{align}
\mathcal{G}(x,l)^{+}_{g}\gg_g \frac{1}{C}\frac{\frac{x}{l}}{\log\log (\frac{x}{l})}.\nonumber
\end{align} 
\end{theorem}

\begin{proof}
Clearly, we can write
\begin{align}
\mathcal{G}(x,l)^{+}_{g}&=\# \{n\leq x:g(n)=g(n+l),~\mathrm{for~fixed}~l>0\}\nonumber \\& \geq \#\left \{n\leq x:g(n)=g(n+l),~l|n,~~\mathrm{for~fixed}~l>0\right \}.\label{ge1} 
\end{align}
We observe that 
\begin{align}
\eqref{ge1}&=\#\left \{n\leq x:g(n)=g(n+l),~l|n,~\mathrm{for~fixed}~l>0,n=l\cdot m,~(l,m)=1\right\}\nonumber \\&+\#\left \{n\leq x:g(n)=g(n+l),~l|n,~\mathrm{for~fixed}~l>0,~n=l\cdot m,~(l,m)\neq 1\right\}\nonumber
\end{align}
so that inequality \eqref{ge1} can be further controlled from below by the quantity
\begin{align}
&\geq \#\left \{m\leq \frac{x}{l}:g(m)g(l)=g(l)g(m+1),~(m,l)=(m+1,l)=1,~\text{for~fixed}~l>0\right\}\label{geq1}
\end{align}
Under requirement $g(t)\neq 0$ for all $t\in \mathbb{N}$ such that \begin{align}
1\ll_g\frac{g(n+1)}{g(n)}\ll_g C\log\log n\nonumber
\end{align}
we deduce
\begin{align}
\eqref{geq1}&=\#\left \{m\leq \frac{x}{l}:g(m)=g(m+1),~\text{for~fixed}~l>0\right\}\nonumber \\&=\sum \limits_{\substack{m\leq \frac{x}{l}\\g(m)=g(m+1)}}1\nonumber \\&\gg \frac{1}{C}\sum \limits_{\substack{m\leq \frac{x}{l}\\g(m+1)=g(m)}}\left(\frac{g(m+1)}{g(m)}\right)\left(\frac{1}{\log\log m}\right)\nonumber \\&\gg_g \frac{1}{C}\sum \limits_{m\leq \frac{x}{l}}\frac{1}{\log\log m}\nonumber
\end{align}
and the lower bound follows immediately.
\end{proof}

\begin{remark}
We particularize the result in Theorem \ref{Type1} by varying our arithmetic functions in the following sequel.
\end{remark}

\begin{corollary}
Let $\varphi(n)=\sum \limits_{\substack{m\leq n\\(m,n)=1}}1$ denote the Euler totient function. We have the lower bound 
\begin{align}
\mathcal{G}(x,l)^{+}_{\varphi}\gg e^{-\gamma}\frac{\frac{x}{l}}{\log \log (\frac{x}{l})}.\nonumber
\end{align}
\end{corollary}

\begin{proof}
By Theorem \ref{euler}, the ratio 
\begin{align}
\frac{\varphi(n+1)}{\varphi(n)}\ll e^{\gamma}\log\log n\nonumber
\end{align}
is satisfied and the lower bound follows immediately from Theorem \ref{Type1}.
\end{proof}

\begin{corollary}
Let $\sigma(n):=\sum \limits_{d|n}d$ denote the sum-of-divisors function. We have the lower bound 
\begin{align}
\mathcal{G}(x,l)^{+}_{\sigma}\gg e^{-\gamma}\frac{\frac{x}{l}}{\log \log (\frac{x}{l})}.\nonumber
\end{align}
\end{corollary}

\begin{proof}
By Theorem \ref{sumdivisor}, we have
\begin{align}
\frac{\sigma(n+1)}{\sigma(n)}\ll e^{\gamma}\log\log n.\nonumber
\end{align}
The lower bound follows by using Theorem \ref{Type1}.
\end{proof}

\begin{theorem}\label{Type2}
Let $g:\mathbb{N}\longrightarrow \mathbb{C}$ be a completely multiplicative function with $g(t)\neq 0$ for all $t\in \mathbb{N}$ and suppose that 
\begin{align}
1\ll_g \frac{g(n+1)}{g(n)}\ll_g 1.\nonumber
\end{align}
 We have
\begin{align}
\mathcal{G}(x,l)^{+}_{g}\gg_g \frac{x}{l}.\nonumber
\end{align}
\end{theorem}

\begin{proof}
Clearly, we can write 
\begin{align}
\mathcal{G}(x,l)^{+}_{g}&=\#\{n\leq x:g(n)=g(n+l),~\mathrm{for~fixed}~l>0\}\nonumber \\& \geq \#\left\{n\leq x:g(n)=g(n+l),~l|n,~\text{for~fixed}~l>0,\right \}.\label{ge2}
\end{align}
We observe that 
\begin{align}
\eqref{ge2}&=\#\left \{n\leq x:g(n)=g(n+l),~l|n,~\text{for~fixed}~l>0,~n=l\cdot m,~(m,l)=1\right\}\nonumber \\&+\#\left \{n\leq x:g(n)=g(n+l),~l|n,~\text{for~fixed}~l>0,~n=l\cdot m,~(l,m)\neq 1\right\}\nonumber
\end{align}
so that inequality \eqref{ge2} can be further written in the form
\begin{align}
&=\#\left\{m\leq \frac{x}{l}:g(m)=g(m+1),~\text{for~fixed}~l>0\right\}.\label{geq2} 
\end{align}
Under condition $g(t)\neq 0$ for all $t\in \mathbb{N}$ and 
$$
1\ll_g \frac{g(n+1)}{g(n)}\ll_g 1.
$$
we deduce
\begin{align}
\eqref{geq2}&=\sum \limits_{\substack{m\leq \frac{x}{l}\\g(m)=g(m+1)}}1\nonumber \\& \gg_g\sum \limits_{\substack{m\leq \frac{x}{l}\\g(m+1)=g(m)}}\frac{g(m+1)}{g(m)}\nonumber \\&\gg_g \sum\limits_{m\leq \frac{x}{l}}1\nonumber
\end{align}
and the lower bound follows immediately.
\end{proof}
\bigskip

Theorem \ref{Type2} can be useful in practice. It suggests that multiplicative functions obeying the underlying conditions with the correlation 
\begin{align}
\sum \limits_{n\leq x}g(n)g(n+l)\nonumber
\end{align}
for a fixed $l>0$ can be well approximated by the partial sum of the corresponding square function 
\begin{align}
\sum \limits_{n\leq x}g(n)^2.\nonumber
\end{align}
\bigskip

Now, we extend our result to multiplicative functions whose consecutive ratio grows by a poly-logarithmic power saving of a logarithm. We will make these statements more precise. It is important to note that these results also hold when we replace a multiplicative function with an additive function.

\begin{theorem}\label{Type3}
Let $f:\mathbb{N}\longrightarrow \mathbb{C}$ be an additive function with $f(n)\neq 0$ for all $n\geq 2$ such that
\begin{align}
1\ll_f\frac{f(n+1)}{f(n)}\ll_f \frac{\log n}{(\log \log n)^c}\nonumber
\end{align}
for some $c>0$. We have 
\begin{align}
\mathcal{G}(x,l)^{+}_{f}\gg_f \frac{x}{l}\cdot \frac{(\log \log (\frac{x}{l}))^c}{\log (\frac{x}{l})}.\nonumber
\end{align}
\end{theorem}

\begin{proof}
Clearly, we can write 
\begin{align}
\mathcal{G}(x,l)^{+}_{f}&=\#\{n\leq x:f(n)=f(n+l),~\text{for~fixed}~l>0\}\nonumber \\& \geq \#\left\{n\leq x:f(n)=f(n+l),~l|n,~\text{for~fixed}~l>0\right \}.\label{ge3}
\end{align}
We can write 
\begin{align}
\eqref{ge3}&=\#\left \{n\leq x:f(n)=f(n+l),~l|n,~\text{for~fixed}~l>0,~n=l\cdot m,~(m,l)=1\right\}\nonumber \\&+\#\left \{n\leq x:f(n)=f(n+l),~l|n,~\text{for~fixed}~l>0,~n=l\cdot m,~(l,m)\neq 1\right\}\nonumber
\end{align}
so that under the requirements $f(n)\neq 0$ for all $n\geq 2$ and 
$$
1\ll_f\frac{f(n+1)}{f(n)}\ll_f \frac{\log n}{(\log \log n)^c}
$$
for some constant $c>0$, we deduce 
\begin{align}
\eqref{ge3}&\geq \#\left\{n\leq x:f(m)+f(l)=f(l)+f(m+1),(m,l)=(m+1,l)=1,~\text{for~fixed}~l>0\right\}\nonumber \\&=\#\left\{m\leq \frac{x}{l}:f(m)=f(m+1),~\text{for~fixed}~l>0\right \}\nonumber \\&=\sum \limits_{\substack{m\leq \frac{x}{l}\\f(m)=f(m+1)}}1\nonumber \\& \gg_f \sum \limits_{\substack{m\leq \frac{x}{l}\\f(m+1)=f(m)}}\left(\frac{f(m+1)}{f(m)}\right)\left(\frac{(\log \log n)^c}{\log n}\right)\nonumber \\&\gg \sum \limits_{m\leq \frac{x}{l}}\frac{(\log \log n)^c}{\log n}\nonumber
\end{align}
The claimed lower bound can be deduced from the partial sum.
\end{proof}

\begin{corollary}
Let $\omega(n):=\sum \limits_{p|n}1$ count the number of distinct prime divisors of $n$. We have the lower bound 
\begin{align}
\mathcal{G}(x,l)^{+}_{\omega}\gg \frac{x}{l}\cdot \frac{(\log \log (\frac{x}{l}))}{\log (\frac{x}{l})}.\nonumber
\end{align}
\end{corollary}

\begin{proof}
By Theorem \ref{divisor2}, we obtain
\begin{align}
1\ll \frac{\omega(n+1)}{\omega(n)}\ll \frac{\log n}{\log \log n}.\nonumber
\end{align}
The lower bound is deduced by applying Theorem \ref{Type3}.
\end{proof}
\bigskip

Keeping in mind the different possible growth rates of the ratios of consecutive values of an additive function, we examine the situation where the ratio grows logarithmically.

\begin{theorem}\label{type4}
Let $f:\mathbb{N}\longrightarrow \mathbb{C}$ be a completely additive function with $f(n)\neq 0$ for all $n\geq 2$ such that \begin{align}
1\ll_f\frac{f(n+1)}{f(n)}\ll_f C\log n.\nonumber
\end{align}
We have the lower bound 
\begin{align}
\mathcal{G}(x,l)^{+}_{f}\gg_f \frac{1}{C}\frac{\frac{x}{l}}{\log (\frac{x}{l})}.\nonumber
\end{align}
\end{theorem}

\begin{proof}
We can write 
\begin{align}
\mathcal{G}(x,l)^{+}_{f}&=\#\{n\leq x:f(n)=f(n+l),~\text{for~fixed}~l>0\}\nonumber \\& \geq \#\left\{n\leq x:f(n)=f(n+l),~l|n,~\text{for~fixed}~l>0\right\}\label{ge4}
\end{align}
We observe that 
\begin{align}
\eqref{ge4}&=\#\left\{n\leq x:f(n)=f(n+l),~l|n,~\text{for~fixed}~l>0,~n=l\cdot m,~(l,m)=1\right\}\nonumber \\&+\#\left\{n\leq x:f(n)=f(n+l),~l|n,~\text{for~fixed}~l>0,~n=l\cdot m,~(l,m)\neq 1\right\}\nonumber 
\end{align}
so that under requirement $f(n)\neq 0$ for all $n\geq 2$ and that 
$$
1\ll_f\frac{f(n+1)}{f(n)}\ll_f C\log n
$$
we deduce
\begin{align}
\\&=\#\left \{m\leq\frac{x}{l}:f(m)=f(m+1),~\text{for~fixed}~l>0\right\}\nonumber \\&=\sum \limits_{\substack{m\leq \frac{x}{l}\\f(m)=f(m+1)}}1\nonumber \\&\gg_f \frac{1}{C}\sum \limits_{\substack{m\leq \frac{x}{l}\\f(m+1)=f(m)}}\left(\frac{f(m+1)}{f(m)}\right)\left(\frac{1}{\log m}\right)\nonumber \\&\gg_f \frac{1}{C}\sum \limits_{m\leq \frac{x}{l}}\frac{1}{\log m}.\nonumber
\end{align}
The lower bound can be deduced from the preceding sum.
\end{proof}

\begin{remark}
We provide a particular instance where these results might be useful considering the number of prime divisors with multiplicity function $\Omega(n)$. 
\end{remark}

\begin{corollary}
Let $\Omega(n):=\sum \limits_{p||n}1$ count the number of prime divisors of $n$ with multiplicity. We have the lower bound \begin{align}
\mathcal{G}(x,l)^{+}_{\Omega}\gg (\log 2)\frac{\frac{x}{l}}{\log (\frac{x}{l})}.\nonumber
\end{align}
\end{corollary}

\begin{proof}
By Theorem \ref{multiplicity}, we obtain the upper bound \begin{align}
\frac{\Omega(n+1)}{\Omega(n)}\ll \frac{\log n}{\log 2}.\nonumber
\end{align}
The lower bound can be deduced using Theorem \ref{Type3}.
\end{proof}
\bigskip

The results suggest that the degree of coincidences under a given multiplicative or additive function is largely dependent on the worst growth rate of consecutive values. That is, the smaller the growth rate of the ratios of consecutive values, the larger the set of coincidences, and vice-versa. Although not all cases have been considered in this study, one could easily notice that if we allow a linear growth rate of the ratio of consecutive values, then the set of coincidence must grow at least logarithmic in size. In particular, we deduce the following

 \begin{theorem}
Let $f:\mathbb{N}\longrightarrow \mathbb{C}$ be a completely additive function with $f(n)\neq 0$ for all $n\geq 2$ such that \begin{align}
1\ll_f \frac{f(n+1)}{f(n)}\ll_f Cn.\nonumber
\end{align}
We have the lower bound 
\begin{align}
\mathcal{G}(x,l)^{+}_{f}\gg_f \frac{1}{C}\log (\frac{x}{l}).\nonumber
\end{align}
\end{theorem} 

\footnote{
\par
.}%

\bibliographystyle{amsplain}

\end{document}